\def\ve{\varepsilon}
\def\Cat{\mbox{CAT}}

\def\ma{\mathfrak{M}}
\def\mm{\mathfrak{M}}

\def\G{{\cal G}}\def\P{\mathcal{P}}
\def\iso{{\rm ISO}}

\def\isomo{{\rm ISOM}}

\def\oN{\mathbb N}

\def\oR{\mathbb R}

\documentclass{article}
\usepackage{amsthm} 
\usepackage{graphicx,hyperref} 
\usepackage{amssymb,latexsym,times}
\newtheorem{theorem}{Theorem}
\newtheorem{definition}[theorem]{Definition}
\newtheorem{lemma}[theorem]{Lemma}
\newtheorem{corollary}[theorem]{Corollary}
\newtheorem{example}[theorem]{Example}
\begin{document}

\title{Tracing Internal Categoricity\footnote{I am indebted to John Baldwin for his critical reading and comments on an earlier manuscript of this paper, and  to Roman Kossak for a discussion on models of arithmetic in relation to this paper. I am also indebted to the referees for valuable remarks.}}
\author{Jouko V\"a\"an\"anen\thanks{Supported by the Faculty of Science of the University of Helsinki and grant 322795 of the Academy of Finland.}\\
Department of Mathematics and Statistics,\\ University of Helsinki\\
\and
Institute for Logic, Language and Computation,\\ University of Amsterdam}
\maketitle

\begin{abstract} Informally speaking, the categoricity of an axiom system means that its non-logical symbols have only one possible interpretation that renders the axioms true. 
Although non-categoricity  has  become ubiquitous in the second half of the 20th century whether one looks at  number theory, geometry or analysis,  
the first axiomatizations of such mathematical theories by Dedekind,  Hilbert, Huntington, Peano and Veblen were indeed categorical. A common resolution of the difference between the earlier  categorical axiomatizations and the more modern non-categorical axiomatizations   is that the latter derive their non-categoricity from  Skolem's Paradox and  G\"odel's Incompleteness Theorems, while the former, being second order, suffer from a heavy reliance on  metatheory, where the Skolem-G\"odel phenomenon re-emerges. Using second order meta-theory to avoid non-categoricity of the meta-theory would only seem to lead to an infinite regress. In this paper we maintain that internal categoricity breaks this traditional picture. It applies to both first and second order axiomatizations, although in the first order case we have so far only examples. It does not depend on the meta-theory in a way that would lead to an infinite regress. And it covers the classical categoricity results of early researchers. In the first order case it is weaker than categoricity itself, and in the second order case stronger. We give arguments suggesting that internal categoricity is the ``right" concept of categoricity.
\end{abstract}

\section{Introduction}

The concept of categoricity of a formal theory was introduced with this name by Veblen \cite{zbMATH02653775}, who notes that the name was suggested to him by John Dewey. Veblen introduced this concept in the context of geometry, but it turned out to have a vastly more general applicability. Veblen defines:

\begin{quote}
Inasmuch as the terms {\em point} and {\em order} are undefined one has a right, in thinking of the propositions, to apply the terms in connection with any class of objects of which the axioms are valid propositions. It is part of our purpose however to show that there is {\em essentially only one} class of which the twelve axioms are valid. In more exact language, any two classes $K$ and $K'$ of objects that satisfy the twelve axioms are capable of a one-to-one correspondence such that if any three elements $A$, $B$, $C$ of $K$ are in the order $ABC$, the corresponding elements of $K'$ are also in the order $ABC$. 
Consequently any proposition which can be made in terms of points and order either is in contradiction with our axioms or is equally true of all classes that verify our axioms. 
\ldots 
A system of axioms such as we have described is called {\em categorical} \ldots 
\end{quote}

In modern terminology we would say that the twelve axioms of Veblen are categorical in the following sense:
\begin{definition}
A set of axioms\footnote{For this definition to make sense it does not matter in which formal language the axioms are written as long as the concepts of {\em model} and {\em isomorphism} make sense. In this paper the axioms are first or second order axioms.} is categorical if  any two of its models  are isomorphic. 
 \end{definition}
Veblen points out what we would now call the {\em completeness} of categorical axiom systems: for any given sentence $\phi$ in the vocabulary of the axioms, either $\phi$ follows\footnote{In modern terminology Veblen's concept of ``follows" would be perhaps best called a semantic consequence relation.} from the axioms or $\neg\phi$ follows. 
As Veblen notes, the concept had been already used with a different name by Huntington \cite{zbMATH02659709}, who calls it `sufficiency', and by  Hilbert \cite{zbMATH02656983}, who is less explicit about it. For an excellent survey of the history of categoricity and completeness, we refer to \cite{Awodey2002-STECAC-6}.

If $L$ is a vocabulary, i.e. a set of relation, function and constant symbols, 
an $L$-model (or an $L$-structure) $\mm$ consists of a domain $M$, which is a non-empty set, together with an interpretation of the symbols of $L$ in the set $M$. Respectively an $L$-sentence of first (or second) order logic is a sentence the non-logical symbols of which are in $L$. The interpretation of a relation symbol $R$ is denoted $R^\mm$, the interpretation of a function symbol $f$ is denoted $f^\mm$, and the interpretation of a constant symbol $c$ is denoted $c^\mm$.
Isomorphism $\mm\cong\mm'$ of $L$-structures $\mm$ and $\mm'$ is defined in the usual way. In this paper identity $=$ is considered a logical symbol and therefore we do not explicitly mention whether identity is included in the vocabulary or not. The interpretation of identity is always equality. If $\mm$ is an $L$-structure and $L' \subseteq L$, then $\mm\restriction L'$ denotes the reduct of $\mm$ to the vocabulary $L'$.

The most famous as well as historically the first categorical axiomatization of a mathematical structure  is Dedekind's second order axiomatization of elementary arithmetic. Note that  in the second order setup the arithmetic operations of addition and multiplications are definable from the successor relation.

\begin{example}[Dedekind]\label{first}Let us consider the vocabulary $\{s,0\}$ of the theory of the successor function on the natural numbers.
The second order sentence $N^2(s,0)$:
\begin{equation}\label{dede}
\begin{array}{l}
\forall xy(s(x)=s(y)\to x=y)\ \wedge\\
\forall x\neg s(x)=0\ \wedge\\
\forall X((X(0)\wedge\forall x(X(x)\to X(s(x))))\to\forall xX(x))\end{array}\end{equation}
is categorical with the unique (up to isomorphism) model $(\oN,f,0)$, where $f(x)=x+1$. 
\end{example} 

Since in second order logic we can existentially quantify over the function $s$ and the constant $0$, we can axiomatize the countably infinite domain:

\begin{example}\label{first1}
The second order sentence $I^2=\exists F\exists x N^2(F,x)$
is categorical with the unique (up to isomorphism) model\footnote{Of the empty vocabulary.} $\oN$. 
\end{example}

Example \ref{first} is different from Example~\ref{first1} in the following important respect. If we take two models of $N^2(s,0)$, there is a {\em unique} isomorphism between them. But  if we take two models of the sentence $I^2$ of Example~\ref{first1}, there is an isomorphism (i.e. a bijection) between the models but the isomorphism is by no means unique, because the set $\oN$ permits a continuum of different bijections.

\begin{example}\label{second} We consider the vocabulary $\{R,\varepsilon\}$, where $R$ is unary and $\varepsilon$ is binary.
The second order sentence $P^2$:
\begin{equation}\label{dede2}
\begin{array}{l}
I^2\mbox{ relativized (see below for the definition of relativization) to $R$ }\wedge\\
\forall xy(x\varepsilon y\to R(x))\wedge\\
\forall xy(\forall z(z\varepsilon x\leftrightarrow z\varepsilon y)\to x=y)\wedge\\
\forall X\exists x\forall y(R(y)\to(X(y)\leftrightarrow y\varepsilon x))\end{array}\end{equation}
is categorical with the unique (up to isomorphism) model $(\P(\oN),\in,\oN)$, where $\in$ is the usual membership-relation between elements of $\oN$ and elements of $\P(\oN)$. 
\end{example} 
 
Again we may note that the isomorphisms manifesting the categoricity of $P^2$ are not unique because of the many automorphisms of the $R$-part. However, if the isomorphism is fixed on the $R$-parts, the rest is unique.

\begin{example}[Huntington]\label{reals} The second order axiomatization $R^2$ of the completely ordered field of real numbers, i.e. the conjunction of the first order axioms of ordered fields and the following {\em Least Upper Bound Principle}:
\begin{equation}
\begin{array}{l}
\forall X((\exists xX(x)\wedge\exists y\forall x(X(x)\to x<y))\to\\
\hspace{3mm}\exists y(\forall x(X(x)\to (x<y\vee x=y))\wedge\\
\hspace{5mm}\forall y'(\forall x(X(x)\to (x<y'\vee x=y'))\to (y<y'\vee y=y')))),\end{array}
\end{equation}
is categorical with the unique (up to isomorphism) model $(\oR,+,\cdot,0,1,<)$. The usual proof of the categoricity proceeds by first isolating the natural numbers as $0,1,1+1,1+1+1,\ldots$, then the rationals, and then the reals as the completion of the rationals. The isomorphism between any two models of $R^2$ is unique.
\end{example}

As a final example, let us consider set theory:
\def\ZFC{\mbox{ZFC}}

\begin{example}[Zermelo]\label{zer}
Let $\ZFC^2$ be the conjunction of the second order Zermelo-Fraenkel axioms, obtained from the ordinary  
Zermelo-Fraenkel axioms by replacing the Separation and Replacement Schemas by their second order versions. Models of $\ZFC^2$ are, up to isomorphism, of the form $(V_\kappa,\in)$, where $\kappa>\omega$ is inaccessible. Since there may be many inaccessibles, $\ZFC^2$ is not categorical unless we make large cardinal assumptions. However, as emphasised already by Zermelo \cite{zbMATH02562682}, $\ZFC^2$ is categorical in the weaker sense that if the height of the model is fixed, then there is, up to isomorphism, only one model. This weak form of categoricity of $\ZFC^2$ is sometimes called {\em quasi-categoricity}.

\end{example}

The range of categoricity among second order theories is  extensive. In fact, it is very hard to find a structure which would be not second order characterizable, without using a cardinality argument or  the Axiom of Choice, see \cite{MR3116536}.

\section{Preliminaries}

\def\Res{\mbox{Res}}
\def\rest{\!\restriction\!}

We considered above some examples of categoricity. Let us now set the stage for a more general approach. The remarkable property of second order logic is that it can express the categoricity of its own sentences. To see what this means we have to introduce some notation.

Let us consider  a finite vocabulary $L=\{R_1,\ldots,R_\alpha,f_1,\ldots,f_\beta\}$ and a unary predicate $U$  not in $L$. If $\ma$ is an $L\cup\{U\}$-structure, let the $L$-structure $\ma^{U}$ be the relativization of  $\ma\rest L$  to the interpretation of the predicate $U$. For $\ma^{U}$ to be a legitimate $L$-structure, something has to be assumed about the interpretation of $U$ as well as about  the interpretations of the function symbols in $\ma$. Let $\Res_{\{f_1,\ldots,f_\beta\}}(U)$ be the conjunction of the first order sentences
\begin{equation}\label{res}
\begin{array}{l}
\exists x_1 U(x_1)\\
\forall x_1\ldots x_n((U(x_1)\wedge\ldots\wedge U(x_n))\to U(f_j(x_1,\ldots,x_n))),\\
\end{array}
\end{equation}
where  $j\in\{1,\ldots,\beta\}$.  Certainly,  $\ma$ satisfies $\Res_{\{f_1,\ldots,f_\beta\}}(U)$ if and only if $\ma^{U}$ is an $L$-structure. If $\phi$ is a second order sentence, we let $\phi^{U}$  denote the result of relativizing $\phi$ to $U$. Thus in $\phi^{U}$ all first order quantifiers are restricted to range over elements of $U$, and the second order variables are restricted to range over subsets of $U$, relations over $U$ and functions on $U$. In consequence, if $\phi$ is an $L$-sentence and $\mm\models\Res_{\{f_1,\ldots,f_\beta\}}(U)$, then 
\begin{equation}\label{epow}\ma\models\phi^{U}\iff\ma^{U}\models\phi.
\end{equation}

Let 
 $L'=\{R'_1,\ldots,R'_\alpha,f'_1,\ldots, f'_\beta\}$ be another vocabulary, where the arity of each $R'_i$ is the same as the arity of $R_i$ and the same for the function symbols. We assume $L\cap L'=\emptyset$.  If $\phi$ is a second order sentence, we use  $\phi'$  to denote the result of replacing each $R_i$ by $R'_i$ and $f_j$ by $f'_j$ in $\phi$. Let $U'$ be a new unary predicate symbol not in $L\cup L'\cup\{U\}$.
%
Let $\iso_{L,L'}(F,U,U')$ be the second order sentence which says that the function  $F$ defines an isomorphism between the $L$-part relativized to $U$ and the $L'$-part relativized to $U'$, i.e. $\iso_{L,L'}(F,U,U')$ is the conjunction of the first order sentences
$$\begin{array}{l}
\forall x(U(x)\to U'(F(x)))\\
\forall xy(F(x)=F(y)\to x=y)\\
\forall x(U'(x)\to\exists y(U(y)\wedge F(y)=x))\\
\forall x_1\ldots x_n(R_i(x_1,\ldots ,x_n)\leftrightarrow R_i'(F(x_1),\ldots,F(x_n)))\\
\forall x_1\ldots x_n(F(f_j(x_1,\ldots ,x_n))=f_j'(F(x_1),\ldots,F(x_n))),\\
\end{array}$$ 
where $i\in\{1,\ldots,\alpha\}$ and $j\in\{1,\ldots,\beta\}$.

\begin{definition}For any second order $L$-sentence $\phi$, the second order\footnote{For simplicity, we use same names for second order relation (or function) variables and for (first order) relation and function symbols.} sentence $\Cat_\phi$ of the empty vocabulary is defined as follows:
$$\begin{array}{l}
\Cat_\phi: \hspace{1mm}\forall U U' R_1\ldots R_\alpha R_1'\ldots R'_\alpha f_1\ldots f_\beta f'_1\ldots f'_\beta \Cat^+_\phi.
\end{array}
$$where $\Cat^+_\phi$ is 
$$\begin{array}{l}
(\Res_{\{f_1,\ldots,f_\beta\}}(U)\wedge\Res_{\{f'_1,\ldots,f'_\beta\}}(U')\wedge
\phi^{U}\wedge\phi'^{U'})\to\exists F \ \iso_{L,L'}(F,U,U').
\end{array}
$$
\end{definition}

\begin{lemma}\label{eq}
Suppose $\phi$ is a second order $L$-sentence. The following conditions are equivalent: 
\begin{description}
\item[(C1)] $\phi$ is categorical.
\item[(C2)] $\Cat_\phi$ is valid in the empty vocabulary.
\item[(C3)] $\Cat^+_\phi$ is valid in the vocabulary $L\cup L'\cup\{U,U'\}$.

\end{description}
\end{lemma}

\begin{proof}
 To see this, suppose first (C1). We show that $\Cat_\phi$ is valid. Suppose $M$ is an arbitrary non-empty set. We can consider $M$ the domain of a model of the empty vocabulary i.e. the vocabulary of $\Cat_\phi$. Let $\ma$ be the $L\cup L'\cup\{U,U'\}$-structure resulting from interpreting the non-logical symbols of $L\cup L'\cup\{U,U'\}$ in the domain $M$ in some arbitrary way. Let us assume $\ma$ satisfies $\Res_{\{f_1,\ldots,f_\beta\}}(U)\wedge\Res_{{\{f'_1,\ldots,f'_\beta\}}}(U')\wedge\phi^{U}\wedge\phi'^{U'}$. Thus, by (\ref{epow}), $(\ma\rest L\cup\{U\})^{U}$ is an $L$-structure satisfying $\phi$. Respectively, $(\ma\rest L'\cup\{U'\})^{U'}$ is an $L'$-structure satisfying $\phi'$. Changing vocabularies, the categoricity of $\phi$ yields $\ma\models\exists F \ \iso_{L,L'}(F,U,U')$. Clearly, (C2) implies (C3). Finally, assume (C3). Let $\mm_1$ and $\mm_2$ be two $L$-models of $\phi$. By changing the vocabulary, we can translate $\mm_2$ into an $L'$-model $\mm_2'$ of $\phi'$. Since $L\cap L'=\emptyset$, we can form an $L\cup L'\cup\{U,U'\}$-structure $\mm$ such that $(\ma\rest L\cup\{U\})^{U}=\mm_1$ and
$(\ma\rest L'\cup\{U'\})^{U'}=\mm_2'$. By (C3),
$\mm\models \exists F \ \iso_{L,L'}(F,U,U')$, whence $\mm_1\cong\mm_2$.

\end{proof}

It should be noted that even though $\Cat_\phi$ has no non-logical symbols,  there is a marked difference between $\Cat_\phi$ being true in a model (of the empty vocabulary) and it being valid in all models (of the empty vocabulary). There is nothing surprising in this. Truth in a model is, in logic in general, very far from being equivalent to truth in all models. However, the case of $\Cat_\phi$ is a very special one.

The truth of 
$\Cat_\phi$ in a model of (the empty vocabulary) of cardinality $\kappa$ means the same as categoricity of $\phi$ in models of cardinality $\le\kappa$. In other words, despite its appearance, the sentence $\Cat_\phi$ does not state the categoricity of $\phi$, because of size limitations. It is only the proposition  \emph{$\Cat_\phi$ is valid} which states it. 

For the categoricity of $\phi$ it obviously suffices that $\Cat_\phi$ has arbitrarily large models. If we let $\kappa$ denote the Hanf-number\footnote{The Hanf number of a logic is the least cardinal such that if a sentence of the logic has a model of size at least $\kappa$, it has arbitrarily large models.} of second order logic then it suffices that $\Cat_\phi$ is true in a model of cardinality $\kappa$ of the empty vocabulary. Since the Hanf-number of second order logic is less than the first extendible cardinal (\cite{MR0295904}), for $\phi$ to be categorical it suffices that $\Cat_\phi$ has a model of size at least the first extendible cardinal.

\section{Internal categoricity}

The idea behind internal categoricity is the observation that in familiar cases of categorical second order sentences $\phi$ the sentence $\Cat_\phi$ is not only valid but even provable. This is remarkable because  second order logic does not have a similar Completeness Theorem as first order logic. So there is no a priori reason why a valid sentence would be provable. 

Following \cite{MR0351742} we include in the axioms of second order logic the following schema:
Suppose $L$ is a vocabulary. The {\em $L$-Axiom Schema of Comprehension} is the following set of second order sentences:
\begin{equation}\label{CA}
\begin{array}{l}
\forall x_1\ldots x_n\forall X_1\ldots X_k\exists X\forall y_1\ldots y_m(X(y_1,\ldots,y_m)\leftrightarrow\\
\hspace{15mm}\phi(x_1,\ldots,x_n,X_1,\ldots, X_k,y_1,\ldots, y_m)),
\end{array}
\end{equation}
where $\phi(x_1,\ldots,x_n,X_1,\ldots, X_k,y_1,\ldots, y_m)$ is an arbitrary second order $L$-formula not containing $X$ free. By provability in second order logic we mean provability from the axioms of second order logic including  (\ref{CA}). The schema (\ref{CA}) is non-trivial even when $\phi(x_1,\ldots,x_n,X_1,\ldots, X_k,y_1,\ldots, y_m)$ is first order. 

\begin{definition}\label{intc}
Suppose $L$ is a vocabulary. A second order $L$-sentence $\phi$ is {\em internally categorical} if the sentence $\Cat_\phi$ (or equivalently, $\Cat^+_\phi$) is provable in second order logic.
\end{definition}

Since provable sentences are valid, internal categoricity implies categoricity. Moreover, internal categoricity can be established for important categorical sentences:

\begin{theorem} The following second order sentences are internally categorical:
\begin{enumerate}
\item The  axiomatization $N^2$ of $(\oN,s,0)$. 
(\cite{Parsons1990-PARTSV}, \cite{Feferman1995-FEFPFO}, \cite{MR3001547}, \cite{Walmsley2002-WALCAI-3}).
\item The  axiomatization $I^2$ of $\oN$ in the empty vocabulary.
\item The  axiomatization $P^2$ of $(\P(\oN),\in,\oN)$.
\item The  axiomatization $R^2$ of $(\oR,+,\cdot,0,1,<)$.
\end{enumerate}
\end{theorem}

\begin{proof} For a soft proof one can use Theorem~\ref{comp} below.
In all cases the proof is the classical proof of categoricity. When one examines the classical proof carefully, nothing beyond the axioms of second order logic is used. Of course, (\ref{CA}) is heavily used, although not in its full generality.
\end{proof}

The idea of the internal categoricity of $N^2$ appears already in \cite{Parsons1990-PARTSV}. In fact, when categoricity of a sentence $\phi$ is defined in \cite{zbMATH03022982} (English translation \cite{MR736686}), it is not defined as the truth of $\Cat_\phi$, but as the provability of $\Cat_\phi$ in the simple theory of types. We go further and require provability in second order logic itself. The phrase `internal categoricity' was introduced in the context of arithmetic in \cite{Walmsley2002-WALCAI-3}. It was advocated more generally in \cite{Vaananen2012-VNNSOL}, with more details in \cite{MR3326591}. We refer to \cite{Button2016-BUTSAC} for a recent discussion on this concept. 

The point of internal categoricity is that it establishes categoricity without requiring a set-theoretic background (meta-)theory. Internal categoricity of, say $N^2$ or $R^2$, is an ``internal" property of second order logic itself. Likewise, the proof of $\Cat_\phi$ gives finite evidence of the categoricity of $\phi$. Again, there is no need to refer to set theory or ``full semantics" (see below). The appealing contention that the categoricity of second order axiomatizations of familiar mathematical structures requires ``full" second order logic, with set-theoretical semantics, would be simply wrong.

Although second order logic does not have a Completeness Theorem in the same sense as first order logic, there is a more general concept of semantics, due to Henkin \cite{MR0036188}, which permits a Completeness Theorem.  In this more general semantics we do not let bound second order variables, such as $U,U',R_i,R_i',g_i,f'_i,F$ in $\Cat_\phi$ above, range over {\em all} possible relations (and functions) over the given domain $M$ but just over elements of a set $\G$ of relations (and functions) on $M$. This is in line with the development in first order logic, made completely explicit in the proof of G\"odel's Completeness Theorem, to consider arbitrary models, with arbitrary relations, functions and elements to interpret the non-logical symbols, even if one is interested in number theory, where the functions symbols $+$ and $\times$  have a clear intended meaning. In second order logic the intended meaning of the range of second order variables (for subsets of the domain) is the entire power-set of the domain, but we allow the more general set $\G$ as the range in order to achieve a Completeness Theorem.

\begin{definition}[\cite{MR0036188}]
Suppose $L$ is a vocabulary. The combination $(\ma,\G)$ of an $L$-structure $\ma$ and a set $\G$ satisfying the $L$-Axiom Schema of Comprehension is called a {\em Henkin $L$-model}. 
\end{definition}

Intuitively speaking the set $\G$ of a Henkin $L$-model $(\mm,\G)$ is closed in the sense that any set or relation that can be defined from the interpretation of the symbols of $L$ by means of quantification over $\G$ is actually already in $\G$. The schema (\ref{CA}) is quite strong even in the case that $L=\emptyset$.

It goes without saying that the set of {\em all} subsets and relations of a given domain $M$ satisfies (\ref{CA}). Likewise, it is obvious that in any Henkin model $(\mm,\G)$ the set  $\G$ is a Boolean algebra which contains all finite subsets of $M$, as well as the interpretations of the symbols of $L$. Because of the strong impredicativity\footnote{The relation $X$ that is defined by $\phi(x_1,\ldots,x_n,X_1,\ldots, X_k,y_1,\ldots, y_m)$ occurs in the range of the universally and existentially quantified second order variables in $\phi(x_1,\ldots,x_n,X_1,\ldots, X_k,y_1,\ldots, y_m)$.} of (\ref{CA}) it is a non-trivial task to construct Henkin models satisfying  (\ref{CA}). The method of Henkin \cite{MR0036188} yields  examples of such $(\mm,\G)$ with countable $\G$ and $M$. 

The point of Henkin models is:
 
\begin{theorem}[\cite{MR0036188}]\label{comp}
Suppose $L$ is a vocabulary. A second order $L$-sentence is provable if and only if it is valid in all Henkin $L$-models.
\end{theorem}

By combining Definition~\ref{intc} and Theorem~\ref{comp} we get:

\begin{theorem}\label{ic}
Suppose $L$ is a vocabulary. A second order $L$-sentence $\phi$ is {\em internally categorical} if the sentence $\Cat_\phi$ is true in all Henkin models.
\end{theorem}

Theorem~\ref{ic} explains why internal categoricity is called ``internal":  It is {\em internal} in the sense of being internal to each $\G$, namely, it is required that if the interpretations of the symbols of the vocabulary are in $\G$, then so is the interpretation of $F$. Nothing is claimed about interpretations outside $\G$. Since Henkin semantics satisfies the Compactness Theorem, the sentences $N^2$ and $R^2$ are not categorical {\em across} Henkin models, i.e. there are Henkin models of e.g. $N^2$ which are non-standard. For every fixed $\G$ there is a unique model of $N^2$, but different $\G$ may give rise to non-isomorphic models.

For second order set theory $\ZFC^2$ internal categoricity holds only in the weaker sense of quasi-categoricity explained in Example~\ref{zer}. However, for example, for  $$\ZFC^2+\mbox{``there are no inaccessible cardinals $>\omega$"}$$ internal categoricity can be proved \cite{MR3326591}.

Our list of internally categorical second order sentences is only the beginning. There are many more examples. In fact, it would be rather surprising if there were an example of a mathematical structure, constructed without the use of Axiom of Choice, a cardinality argument, or the enumeration techniques \`a la G\"odel, which did not have an internally categorical second order characterization.

In sum, the classical categoricity results of second order logic, due to Dedekind, Veblen and others, hold in the stronger sense of internal categoricity. The advantage of internal categoricity is that it has always a finite proof in the formal language of second order logic itself. Therefore it does not depend on set-theoretical meta-theory in the same sense as ordinary categoricity.

\section{First order internal categoricity---mapping the landscape}

Because of the extraordinary ability of second order logic to express the categoricity of  its own sentences, it would seem that internal categoricity, based on the assumption of the categoricity statement holding in all Henkin models, has no role in first order logic. However, this is far from true. We will show that first order arithmetic and first order set theory have {\em a form} of  internal categoricity, reminiscent of the internal categoricity of their second order cousins. But we do not (know how to) define a general concept of first order internal categoricity. Rather, we are mapping the landscape in search of a good definition. At the moment we have just a few examples.

First order theories are not categorical {\em per se}. Their defining characteristic is the existence of models of all infinite cardinalities and in many important cases also the existence of non-standard models in each infinite cardinality separately. First order internal categoricity has to circumvent these undeniable facts. 

Let us start with first order Peano arithmetic.  Let $P(+,\cdot)$ be  the infinite set of Peano axioms with the Schema of Induction
\begin{equation}\label{induction}
\begin{array}{l}
\forall x_1\ldots x_n((\phi(0,x_1,\ldots, x_n)\wedge\\
\forall y(\phi(y,x_1,\ldots, x_n)\to \phi(y+1,x_1,\ldots, x_n)))\to\\
\forall y\phi(y,x_1,\ldots, x_n)),
\end{array}
\end{equation}
where $0$ and $1$ are defined terms denoting the identity elements of $+$ and $\cdot$, respectively. Ordinarily we assume that the formula $\phi(y,x_1,\ldots, x_n)$ of (\ref{induction}) is any first order formula of the vocabulary $\{+,\cdot\}$ of arithmetic.  However, given a vocabulary $L$, let us use $P(+,\cdot,L)$ to denote the extension of $P(+,\cdot)$ where the Induction Schema (\ref{induction}) is formulated for first order formulas $\phi(y,x_1,\ldots, x_n)$ of 
the vocabulary $\{+,\cdot\}\cup L$.

\begin{theorem}[Internal categoricity of Peano arithmetic]
 Suppose $\ma$ is a model of 
 \begin{equation}\label{woius}
\begin{array}{l}
P(+,\cdot,\{+',\cdot'\})\cup P(+',\cdot',\{+,\cdot\}).
\end{array}
\end{equation}
Then there is $\pi:\ma\!\restriction\!\{+,\cdot\}\cong \ma\!\restriction\!\{+',\cdot'\}$. Moreover, the mapping $\pi$ is first order definable on $\ma$.
\end{theorem}

\begin{proof}
Note that $\ma$ may very well be a non-standard model of any infinite cardinality. 
Let $0$ and $1$ be the first elements of $\ma\!\restriction\!\{+,\cdot\}$, and respectively $0'$, $1'$.
Let $\psi(x,u,v)$ say that $x$ codes, using $+$ and $\cdot$, an initial segment $I$ with the last element $u$, of $\ma\!\restriction\!\{+,\cdot\}$, an initial segment $I'$ with the last element $v$, of 
$\ma\!\restriction\!\{+',\cdot'\}$, and a function $f:I\to I'$ such that $f(0)=0'$, $f(y+1)=f(y)+'1'$ for all $y\in I\setminus\{u\}$, and $f(u)=v$. 
Let $\phi(u,v)$ be the formula $\exists x\psi(x,u,v)$.

First we show by induction on $a$ that for every $a\in M$ there is $b\in M$ such that $\ma\models\phi(a,b)$. If $a=0$, then we need only code the pair $(0,0')$ to make $\phi(0,0')$ true. Suppose then $\phi(a,b)$ is true. Let $x,I,I'$ and $f$ be as in $\psi(x,a,b)$. Clearly,  $f(1)=1'$. Let $I_1$  be $I$ with $a+1$ added, and similarly $I'_1$. Extend $f$ to $f_1$ by mapping the last element $a+1$ of $I_1$ to the last element $b+'1'$ of $I'_1$. 
By coding $I_1,I'_1$ and $f_1$ we get $x_1$ such that $\psi(x_1,a+1,b+'1')$. Thus $\phi(a+1,b+'1')$ is true. By the Induction Schema, since $\ma$ satisfies $P(+,\cdot,\{+',\cdot'\})$, for every $a\in M$ there is $b$ such that $\ma\models\phi(a,b)$. Analogously, since $\ma$ satisfies $P(+',\cdot',\{+,\cdot\})$, we can prove by induction in $\ma\!\restriction\!\{+',\cdot'\}$ that for every $b\in M$ there is $a$ such that $\ma\models\phi(a,b)$.

Next we  show by induction on $a$ that if $\phi(a,b)$ and $\ma\models\phi(a,b')$, then $b=b'$. Suppose first $\phi(0,b)$ and $\phi(0,b')$. Let $x,I,I'$ and $f$ be as in $\psi(x,0,b)$, and $x_1,I_1,I_1'$ and $f'$  as in $\psi(x_1,0,b')$. Then $b=f(0)=0'$ and $b'=f'(0)=0'$. Thus $b=b'$. Suppose then $\phi(a+1,b)\wedge\phi(a+1,b')$ is true. Let $x,I,I'$ and $f$ be as in $\psi(x,a+1,b)$, and $x_1,I_1,I_1'$ and $f'$  as in $\psi(x_1,a+1,b')$. Since we assume that $a$ satisfies the claim, $f(a)=f'(a)$. Hence $b=f(a+1)=f(a)+'1'=f'(a)+'1'=f'(a+1)=b'$.  By the Induction Schema, since $\ma$ satisfies $P(+,\cdot,\{+',\cdot'\})$, for every $a\in M$ there is a unique $b$ such that $\ma\models\phi(a,b)$. 

Note that if $\phi(a,a')$ and $\phi(b,b')$ with $x,I,I'$ and $f$  witnessing $\psi(x_1,a,a')$ and 
$x_1,I_1,I_1'$ and $f'$  witnessing $\psi(x_1,b,b')$, then $f(y)=f'(y)$ for $y\in I\cap I'$. This is because the claim clearly holds for $y=0$, and if it holds for $y$, then $f(y+1)=f(y)+'1'=f'(y)+'1'=f'(y+1)$. So the claim follows by induction in $\ma\!\restriction\!\{+,\cdot\}$ as $\ma$ satisfies $P(+,\cdot,\{+',\cdot'\})$.

Finally, we show that the relation $F=\{(a,b)\in M\times M : \ma\models\phi(a,b)\}$ is an isomorphism $\ma\!\restriction\!\{+,\cdot\}\to \ma\!\restriction\!\{+',\cdot'\}$. Suppose $\alpha,\beta$ and $\gamma$ are   such that $\alpha+\beta=\gamma$. Suppose  $x,I,I'$ and $f$  witness $\psi(x,\gamma,f(\gamma))$. By the above, $f\subseteq F$. We use induction on $\beta$ in $\ma\!\restriction\!\{+,\cdot\}$, to prove that $f(\alpha)+'f(\beta)=f(\gamma)$, appealing to the fact that $\ma$ satisfies $P(+,\cdot,\{+',\cdot'\})$.  If $\beta=0$, then $\alpha=\gamma$ and $$f(\alpha)+'f(\beta)=f(\alpha)+'f(0)=f(\alpha)+'0'=f(\alpha)=f(\gamma).$$ Assume then $\beta=\delta+1$. Now 
$$f(\alpha)+'f(\beta)=f(\alpha)+'f(\delta+1)=f(\alpha)+'f(\delta)+'1'=$$
$$=f(\alpha+\delta)+'1'= f(\alpha+\delta+1)=f(\alpha+\beta)=f(\gamma).$$

Suppose then $\alpha,\beta$ and $\gamma$ are   such that $\alpha\cdot\beta=\gamma$. We use induction on $\beta$ in $\ma\!\restriction\!\{+,\cdot\}$ to prove that $f(\alpha)\cdot'f(\beta)=f(\gamma)$, appealing to the fact that $\ma$ satisfies $P(+,\cdot,\{+',\cdot'\})$.  If $\beta=0$, then $\gamma=0$ and $$f(\alpha)\cdot'f(\beta)=f(\alpha)\cdot'f(0)=f(\alpha)\cdot'0'=0'=f(\gamma).$$ Assume then $\beta=\delta+1$. Now 
$$f(\alpha)\cdot'f(\beta)=f(\alpha)\cdot'f(\delta+1)=f(\alpha)\cdot'(f(\delta)+'1')=$$
$$=f(\alpha)\cdot'f(\delta)+'f(\alpha)\cdot' 1'=f(\alpha\cdot\delta)+'f(\alpha)=f(\alpha\cdot\delta+\alpha)=f(\gamma).$$

 \end{proof}

Since $\pi$ above is definable and ``$\pi:\ma\!\restriction\!\{+,\cdot\}\cong \ma\!\restriction\!\{+',\cdot'\}$" is a first order statement---let us denote it $\isomo(+,\cdot,+',\cdot')$---we can rephrase the above theorem as a theorem of first order logic: 

\begin{corollary}[Internal categoricity of Peano arithmetic in first order logic]
 
 \begin{equation}\label{woius2}
\begin{array}{l}
P(+,\cdot,\{+',\cdot'\})\cup P(+',\cdot',\{+,\cdot\})\vdash \isomo(+,\cdot,+',\cdot').
\end{array}
\end{equation}
\end{corollary}

The internal categoricity of first order arithmetic may seem simply false on the basis that there are a continuum of  non-isomorphic countable models of Peano's axioms. 
 Indeed, suppose $(\oN,+',\cdot')$ is a non-standard model of Peano arithmetic such that $+'$ and $\cdot'$ are $\Delta^0_2$-definable (by \cite{MR0195725}) as relations. Then $(\oN,+,\cdot,+',\cdot')$ satisfies 
$P(+,\cdot,\{+',\cdot'\})$ and $P(+',\cdot')$, but there is no reason to think that it satisfies $P(+',\cdot',\{+,\cdot\})$, for when the $\Delta^0_2$-definable relations $x+'y=z$ and $x\cdot'y=z$ are constructed, only induction for formulas involving $+'$ and $\cdot'$ are (or can be) considered. We know that $(\oN,+',\cdot')$ cannot satisfy $N^2$, so it is interesting to note that the failure of the Induction Axiom of (\ref{dede}) in $(\oN,+',\cdot')$ is manifested by the failure to satisfy $P(+',\cdot',\{+,\cdot\})$.

We now move to set theory. The first order analogue of $ZFC^2$ is the Zermelo-Fraenkel axiom system $ZFC$. Let us denote the vocabulary of this theory  $\{\ve_1\}$.  
To formulate internal categoricity of $ZFC$ we introduce another binary symbol $\ve_2$ and compare $ZFC$ in the vocabulary $\{\ve_1\}$, denoted $ZFC(\ve_1)$, and $ZFC$ in the vocabulary $\{\ve_2\}$, denoted $ZFC(\ve_2)$. 
If we allow symbols from a new vocabulary $L$ to occur in the Separation Schema and the Replacement Schema of $ZFC(\ve_1)$, we denote the extended theory $ZFC(\ve_1,L)$. Similarly $ZFC(\ve_2,L)$.

The second order theory $ZFC^2$ is not categorical per se. One has to fix the (inaccessible) cardinality of the model. In the first order context we accomplish this by assuming that the two models have the same domain.

\begin{theorem}[Internal categoricity of first order $ZFC$, \cite{JVZer}]\label{zfc}
 Suppose $\ma$ is a model of 
$$\begin{array}{l}
ZFC(\ve_1,\{\ve_2\})\cup ZFC(\ve_2,\{\ve_1\}).
\end{array}
$$
Then there is ${\pi}:\ma\!\restriction\!\{\ve_1\}\cong \ma\!\restriction\!\{\ve_2\}$. Moreover, the mapping ${\pi}$ is first order definable on $\mm$.

\end{theorem}

We may again ask, what about the incompleteness of $ZFC$?  If we start with a countable transitive model  $(M,\ve_1)$  of $ZFC$, we may use inner models or forcing to get other countable model $(M',\ve_2)$ of $ZFC$ with the same order-type of the class of ordinals. By applying a re-enumeration we may assume $M'=M$. Numerous set-theoretic statements, such as the Continuum Hypothesis $CH$ or the Souslin Hypothesis $SH$, may be true (or false) in $(M,\ve_1)$ but false (or true) in $(M,\ve_2)$. How is this utter incompleteness consistent with the idea of internal categoricity? 
The answer to this riddle is the following:
An important part of the proof that an inner model is a model of $ZFC$, or that a forcing extension preserves all the axioms of $ZFC$, is checking the validity of the Separation Schema and the Replacement Schema. Just as in the case of the Henkin construction for a non-standard model, we simply cannot allow extra non-logical symbols in the formulas of the Separation Schema and the Replacement Schema. If, on the other hand, we aim at the second order version $ZFC^2$ with bound relation variables instead of formulas in the Separation the Replacement Axioms, we  cannot get the independence results. 

Both arithmetic and set theory manifest internal categoricity via a definable isomorphism. This means that the statement of internal categoricity is first order expressible just as categoricity of a second order sentence is expressible in second order logic itself. 

\section{What is the value of internal categoricity?}

Categoricity is such a beautiful and simple concept that one hesitates to modify it in any way. However, categoricity is usually defined in set theory and may also {\em depend} on set theory.  We have argued that internal categoricity is a more basic concept that does {\em not} depend on set theory. 

In the second order case internal categoricity is stronger than categoricity and can be rephrased simply as provability of the statement of categoricity, because categoricity can be expressed in the second order language itself. In the first order case the situation is more complex. For one thing, we cannot directly express categoricity by first order means. Conceivably we could still consider the provability of the set-theoretic statement of categoricity but there are no non-trivial examples among first order theories.  Instead, the internal categoricity approach is to require  categoricity only as far as an alternative model can be `seen'. In the case of first order internal categoricity we offer an alternative model, such as  $+',\cdot',\epsilon'$, to be `seen' by adding it to the language and allowing it to be used in the axioms. If the axioms remain true in the extended language the power of the axioms generates an isomorphism.  

Loosely speaking, first order arithmetic and set theory are categorical in their own vicinity, among models they can see, but they are not categorical if models are allowed to be constructed from `outside' at will. The difference to second order arithmetic and set theory is that the bound relation variables of the latter reach all possible models, in the case of Henkin models all  models coded by the Henkin model. In a sense, measured by the Henkin model, they see all possible models, all possible models are in their vicinity, and still their axioms are true. That explains why they are categorical. Still the ultimate reason for their categoricity is that they possess the property of being  internally categorical. In this sense internal categoricity is the ``right" concept of categoricity, ``right" in the sense of being the most fundamental.

\end{document}